\documentclass[12pt,reqno]{amsart}

\usepackage[T1]{fontenc}
\usepackage{enumitem}
\setlist[enumerate,1]{wide, label=\emph{(\roman*)}, labelindent=0pt, itemsep=.05in}
\usepackage{hyperref}
\hypersetup{
  colorlinks   = true,
  citecolor    = blue,
  linkcolor    = blue
}

\usepackage{amsmath,amsfonts,amsthm,amssymb,color,tikz, comment}
\usepackage{mathrsfs}

\usepackage{pdfsync}
\usepackage[font={scriptsize}]{caption}

\usepackage[left=1in, right=1in, top=1.1in,bottom=1.1in]{geometry}
\newcommand{\cre}{\color{red}}

\definecolor{dg}{rgb}{0, 0.5, 0}
\definecolor{dp}{rgb}{0.50, 0, 0.40}

\newcommand{\E}{\mathbb E}
\newcommand{\R}{\mathbb R}
\newcommand{\N}{\mathbb N}
\newcommand{\PP}{\mathbb P}

\newcommand{\Z}{\mathbb Z}
  
\newcommand{\al}{\alpha}

\newcommand{\ep}{\varepsilon}

\newcommand{\si}{\sigma}

\newtheorem{theorem}{Theorem}[section]

\newtheorem{definition}[theorem]{Definition}

\newtheorem{lemma}[theorem]{Lemma}

\newtheorem{proposition}[theorem]{Proposition}

\theoremstyle{remark}
\newtheorem{remark}[theorem]{Remark}

\theoremstyle{remark}
\newtheorem{example}[theorem]{Example}

\newcommand{\bean}{\begin{eqnarray*}}
\newcommand{\eean}{\end{eqnarray*}}
\newcommand{\ben}{\begin{enumerate}}
\newcommand{\een}{\end{enumerate}}
\newcommand{\beq}{\begin{equation}}
\newcommand{\eeq}{\end{equation}}

\begin{document}

\title[Fourier dimension of fractional Brownian graphs]
{On the Fourier dimension\\ of fractional Brownian graphs}

\author[C. Y. Lee \and S. Tindel]
{Cheuk Yin Lee \and Samy Tindel}

\newcommand{\Addresses}{{
  \bigskip
  \footnotesize

 \noindent
  \textsc{C. Y. Lee}: School of Science and Engineering, The Chinese University of Hong Kong (Shenzhen), Longgang, Shenzhen, Guangdong, 518172, P.R. China\\
   \textit{E-mail address}: \texttt{leecheukyin@cuhk.edu.cn}

  \medskip

  \noindent
 \textsc{S.~Tindel}: Department of Mathematics,
Purdue University, West Lafayette, IN 47907, USA.\\
  \textit{E-mail address}: \texttt{stindel@purdue.edu}

}}

\maketitle

\begin{abstract}
In this note we prove that the Fourier dimension of the graph $G(B)$ of a fractional Brownian motion $B$ with Hurst parameter $H\in(0,1/2)$ is equal to 1. This finishes to solve a conjecture by Fraser and Sahlsten~\cite{FS18}. It also yields an exact formula for the gap $\dim_{\rm H}(G(B)) - \dim_{\rm F}(G(B))$ between the Hausdorff dimension and the Fourier dimension of $G(B)$. The proof is based on an intricate combinatorics procedure for multiple integrals related to the covariance function of the fractional Brownian motion.
\end{abstract}

\tableofcontents

\section{Introduction}

Let $A \subset \R^n$ be a Borel set and $\mathcal{P}(A)$ denote the set of all Borel probability measures on $\R^n$ with $\mu(A) = 1$.
For any $\mu \in \mathcal{P}(A)$, the Fourier transform of $\mu$ is defined by
\begin{align*}
	\hat{\mu}(\xi) = \int_{\R^n} e^{-2\pi i \xi \cdot x} d\mu(x), \quad \xi \in \R^n.
\end{align*}
The Fourier dimension of $A$ is then defined by
\begin{equation}\label{eq:def-fourier-dim}
	\dim_{\rm F} A = \sup\left\{ \beta \in [0, n] : \exists\,\mu \in \mathcal{P}(A), |\hat{\mu}(\xi)| \lesssim |\xi|^{-\beta/2} \right\},
\end{equation}
where $f(x) \lesssim g(x)$ means that there exists a constant $C<\infty$ such that $f(x) \le C g(x)$ for all $x$.
The relevance of this definition is ensured by a number of geometric properties. Among those, let us mention the following:

\begin{enumerate}[label=\textbf{(\roman*)}]
    \item The Fourier dimension provides a lower bound of the Hausdorff dimension. Specifically, for a measurable set $A \subset \mathbb{R}^n$ we have
    \begin{equation}\label{a1}
        \dim_{\rm F} A \le \dim_{\rm H} A.
    \end{equation}

    \item For manifolds in $\mathbb{R}^n$ with non-vanishing curvature, inequality~\eqref{a1} becomes in fact an equality.

    \item The notion of Fourier dimension is especially relevant for stochastic processes, since they can be easily related to local time properties.
\end{enumerate}
The concept of Fourier dimension also encodes a number of challenging arithmetic properties. We refer to~\cite{LL25} for more details about these aspects of the theory.

With those motivations in mind, observe that one of the key questions for a random set $A$ is to know if its Hausdorff dimension coincides with its Fourier dimension. A set with such a property is called a \emph{Salem set}, a name coined by Kahane in \cite{K93}. Whenever $A$ is not a Salem set, a natural question is then to quantify the difference
$\dim_{\rm H}(A) - \dim_{\rm F}(A)$.
There has been a recent burst of studies carrying out this type of program for the most common (mostly Gaussian) processes. Among the classes of sets which have been considered, let us mention the following:
\begin{enumerate}[label=\textbf{(\arabic*)}]
    \item Images of compact sets $A \subset [0,1]$ by a process $X : [0,1] \to \mathbb{R}^n$. For a one-dimensional fractional Brownian motion $B^H$ with parameter $H \in (0,1)$, one gets that $B^H(A)$ is a Salem set with
    \begin{equation*}
        \dim_{\rm F}\big(B^H(A)\big) = \min\left\{1, H^{-1}\dim_H(A)\right\}.
    \end{equation*}
    This result was established in \cite{K85}, with further generalizations to Gaussian fields in \cite{SX06}.
    
    \item The level sets of a fractional Brownian motion were considered in \cite{Mu}. Specifically, it is proved therein that the zero set of a fractional Brownian motion is almost surely a Salem set.
    
    \item The graph of a function $f : [0,1] \to \mathbb{R}$ is defined by
    \begin{equation}\label{a2}
        G(f) = \left\{ (t, f_{t}) ; \,  t \in [0,1] \right\}.
    \end{equation}
    For most natural stochastic processes, this type of set is not Salem. Indeed, a general result in \cite{FOS14}  asserts that for any continuous function $f$ one has
    \begin{equation}\label{a3}
        \dim_{\rm F} G(f) \le 1 \, ,
    \end{equation}
    while for a fractional Brownian motion, the article \cite{A77} states that $\dim_{\rm H}G(B) = 2-H > 1$.
\end{enumerate}

Due to the phenomenon mentioned above, graphs of stochastic processes have attracted a lot of attention in recent years. More specifically, in order to quantify the difference $\dim_{\rm H}(G(B)) - \dim_{\rm F}(G(B))$ for a generic process $B$ and recalling relation~\eqref{a2}, one wishes to prove that $\dim_{\rm F}(G(B))=1$ for most cases of interest. Relevant results in this direction include
Fraser and Sahlsten~\cite{FS18}, who proved that the Fourier dimension of the graph of the standard Brownian motion is almost surely equal to 1, and conjectured that the Fourier dimension of the graph of fractional Brownian motion is also equal to 1. Then
Lai and Lee~\cite{LL25} confirmed part of this conjecture by proving that the Fourier dimension of the graph of fractional Brownian motion with $H \in [1/2, 1)$ is equal to 1.
Whether the same holds for $H\in (0,1/2)$ remained an open problem; see \cite[Open Problem 1.3]{LL25}.
The main result of this paper fully confirms this conjecture:

\begin{theorem}\label{th:main}
The Fourier dimension of the graph of a fractional Brownian motion with Hurst parameter $H\in(0,1)$ is almost surely 1.
\end{theorem}

The proof of Theorem \ref{th:main} relies on an intricate integration by parts procedure involving combinatorics tools.
The strategy employed in our proof is detailed in Section \ref{s:reduction}.
However, one can already highlight a few key points at this stage:
\begin{enumerate}[label=\textbf{(\roman*)}]
\item A proper study of the Fourier dimension of the graph measure can be reduced to an evaluation of the moments of the following quantity (see \eqref{b1} for more context):
\begin{align}\label{hat:mu}
	\hat{\mu}_G(\xi) = \int_0^1 e^{-2\pi i (\xi_1 t + \xi_2 B_t)} dt
	\quad \text{for $\xi = (\xi_1,\xi_2) \in \R^2$.}
\end{align}
The goal is to obtain a precise rate of decay for the moments. In our case we will discover (see~\eqref{E:horizontal} below) that
$\E[ |\hat{\mu}_G(\xi_1,\xi_2)|^{2q} ] \lesssim |\xi_1|^{-q}$.
\item Some scaling properties and moment computations for the fractional Brownian motion $B$ reduce our estimates to the following problem: consider $\varepsilon \in \{-1,1\}^{2q}$, $T>0$ and $\lambda \in \R$.
We wish to extract powers of $T$ and $\lambda$ from an integral $\mathcal{I}^\lambda_T[\varepsilon,G_\varepsilon]$ over the simplex $\Delta_T = \{0<u_1<\cdots<u_{2q}<T\}$:
\begin{align}\label{I1}
	\mathcal{I}^\lambda_T[\varepsilon,G_\varepsilon] = \int_{\Delta_T} e^{-2\pi i\lambda \langle \varepsilon, u\rangle} e^{-\pi \mathrm{Var}(\sum_{i=1}^{2q} \varepsilon_i B_{u_i})} du.
\end{align}
\item Extracting powers from \eqref{I1} will be our main endeavor. It will be achieved in the lengthy and technical Lemma \ref{lem:I:UB}.
Roughly speaking, one starts from \eqref{I1} and performs integration by parts in order to extract negative powers of $\lambda$. Then one has to carefully choose the integrations by parts in order to sort out two problems: (a) The combinatorics related to the variables $\varepsilon_i \in \{-1,1\}$; and (b) The singularities created by the differentiation of $u \mapsto \mathrm{Var}(\sum_{i=1}^{2q} \varepsilon_i B_{u_i})$. Those singularities are particularly challenging in our current case of study $H \in (0,1/2)$.
\end{enumerate}

\noindent
The remainder of the paper details the strategy laid out above, with a particular emphasis on integration by parts.

\section{Preliminary results}

This section is devoted to recalling some basic facts about Fourier dimension computations. They are mostly borrowed from~\cite{LL25} and serve as a basis for our computations in the rough case of a fractional Brownian motion with Hurst parameter $H<1/2$.

\subsection{Reduction to variance estimates}\label{s:reduction}

The processes considered in this paper are generically denoted by $B = \{ B_t \,, t \ge 0\}$.
Those will be fractional Brownian motions with Hurst parameter $H \in (0,1/2)$. 
Let us recall that $B$ is thus a centered Gaussian process defined on a complete probability space $(\Omega, \mathscr{F}, \PP)$, with covariance function
\begin{align}\label{fBm:cov}
	\E[B_s B_t] = \frac12 (|s|^{2H} + |t|^{2H} - |t-s|^{2H})\, , \qquad \text{for $s,t \ge 0$.}
\end{align}
For this process we consider the graph measure on $\R^2$, which is defined by
\begin{align*}
	\mu_G(A) = \int_0^1 {\bf 1}_{\{t \in [0,1]: (t, B_t) \in A\}} dt \, ,
	\qquad \text{for every Borel set $A \subset \R^2$.}
\end{align*}
Note that one can also write
\begin{align}\label{mu_G}
	\mu_G(A) = \iint_A {\bf 1}_{[0,1]}(t) \, \delta_x(B_t)\, dx\, dt,
\end{align}
so that formally $\mu_G$ can be recast as
\begin{align*}
	\mu_G(dt, dx) = {\bf 1}_{[0,1]}(t) \, \delta_x(B_t) \, dt\, dx.
\end{align*}
With this relation in hand, introduce the Fourier transform of $\mu_G$ as
\begin{equation}\label{b1}
	\hat{\mu}_G(\xi) 
	= \int_0^1 \int_\R e^{-2\pi i (\xi_1 t + \xi_2 x)}\mu_G(dt, dx)
	= \int_0^1 e^{-2\pi i (\xi_1 t + \xi_2 B_t)} dt.
\end{equation}
In the sequel we will also consider the image measure of the process $B$.
We recall its definition below.

\begin{definition}
Let $B$ be a fractional Brownian motion, with covariance function given by~\eqref{fBm:cov}.
Recall that the graph measure $\mu_G$ of $B$ is given by \eqref{mu_G}.
Then the image measure of $B$ is defined on Borel sets $F \subset \R$ by
\begin{align}\label{b2}
	\nu(F) = \mu_G([0,1] \times F) = \int_0^1 {\bf 1}_{\{B_t \in F\}} \, dt.
\end{align}
Its Fourier transform $\hat{\nu}$ is given for $\xi_2 \in \R$ by
\begin{align*}
	\hat{\nu}(\xi_2) = \int_0^1 e^{-2\pi i \xi_2 B_t} dt.
\end{align*}
\end{definition}


A standard way of finding Fourier dimension of the image of a stochastic process is through computing moments of the image measure $\nu$.
This method was first used by Kahane \cite{K85}.
General statements for this approach can be found in \cite{DL23, E16, F24}.
In order to prove Theorem \ref{th:main}, we will use the following specific version in Lai and Lee \cite{LL25} that is adapted to our current setting for getting lower bounds on the Fourier dimensions of graphs.

\begin{proposition}\label{pr:Fdim}\cite[Proposition 2.3]{LL25}.
For a generic process $B$, recall that the graph and image measures $\mu_{G}$ and $\nu$ are respectively defined by~\eqref{mu_G} and~\eqref{b2}.
Suppose for every $q \in \N$, there exists $C>0$ such that
\begin{enumerate}
\item (vertical bound) $\E[|\hat{\nu}(\xi_2)|^{2q}]\le C |\xi_2|^{-\gamma_2 q}$ for all $\xi_2 \ne 0$; and
\item (horizontal bound) $\E[|\hat{\mu}_G(\xi_1,\xi_2)|^{2q}] \le C |\xi_1|^{-\gamma_1 q}$ for all $\xi_1 \ne 0$ and $\xi_2 \in \R$.
\end{enumerate}
Then $\dim_{\rm F} G(B) \ge \min\{\gamma_1,\gamma_2\}$ a.s.
\end{proposition}

Within the landmark of Proposition \ref{pr:Fdim}, we thus wish to establish vertical and horizontal bounds for the fractional Brownian motion.
A first remark in this direction is that the vertical bounds are already known (see \eqref{E:vertical} below).
Therefore, we will concentrate our efforts on the horizontal bounds in Proposition \ref{pr:Fdim}.
To this end, we will follow the combinatorial framework in Lai and Lee \cite{LL25}.
Let us first recall some notation and tools introduced in \cite{LL25}, starting from the definition of simplex:
for $q \in \N_+$ and $T>0$, we consider the simplex $\Delta_T$  in $\R^{2q}$ as
\begin{equation}\label{b3}
	\Delta_T = \{ u = (u_1,\dots, u_{2q}) \in \R^{2q} : 0<u_1<\cdots<u_{2q}<T \}.
\end{equation}
We will also use sets of Rademacher type variables given as follows:
\begin{align}\label{A}
	\mathcal{A}_{2q} = \Big\{ \varepsilon=(\varepsilon_1, \dots, \varepsilon_{2q}) \in \{-1,1\}^{2q} : \sum_{j=1}^{2q} \varepsilon_j = 0 \Big\}.
\end{align}
Then for any $\varepsilon \in \mathcal{A}_{2q}$, $\lambda \in \R$, and any integrable function $G$ on $\Delta_T$, define
\begin{align}\label{I}
	\mathcal{I}[\varepsilon, G] = \mathcal{I}^\lambda_T[\varepsilon,G] := \int_{\Delta_{T}} e^{-2\pi i\lambda \langle\varepsilon, u \rangle} G(u) du.
\end{align}
The following formula for the even moments of $\hat{\mu}_G$ has been obtained in \cite{LL25} using Kahane's decomposition and the scaling property of fractional Brownian motion.

\begin{lemma}\cite[Lemma 5.1]{LL25}.
Let $B$ be a fractional Brownian motion with Hurst parameter $H \in (0,1)$, and recall that the Fourier transform of $\mu_{G}$ is given by~\eqref{b1}. Then for any $q\in \N_+$ and $\xi_1, \xi_2 \in \R \setminus\{0\}$, we have
\begin{align}\label{2q:moment}
	\E\left[ |\hat{\mu}_G(\xi_1,\xi_2)|^{2q} \right]
	= \frac{(q!)^2}{|\xi_2|^{2q/H}} \sum_{\varepsilon\in \mathcal{A}_{2q}} \mathcal{I}_{T}^{\lambda}[\varepsilon, G_\varepsilon],
\end{align}
where the parameters $\lambda, T$ and the function $G_{\varepsilon}$ are chosen as
\begin{align}\label{lambda,T}
	\lambda = \frac{\xi_1}{|\xi_2|^{1/H}}, \qquad
	T = |\xi_2|^{1/H} ,\quad \text{ and }\quad
	G_\varepsilon(u) = e^{-\pi\mathrm{Var}(\sum_{i=1}^{2q} \varepsilon_i B_{u_i})}.
\end{align}
\end{lemma}

With formula \eqref{2q:moment} in hand, and keeping in mind that we wish to get the horizontal bound
\[
	\E\left[ |\hat{\mu}_G(\xi_1,\xi_2)|^{2q} \right] \le \frac{C_q}{|\xi_1|^{q\gamma_1}}
\]
with $\gamma_1 = 1$, it is readily checked that it is sufficient to bound the integrals $\mathcal{I}_T^\lambda[\varepsilon,G]$ in \eqref{I} as follows:
\begin{align}\label{a}
	|\mathcal{I}_T^\lambda[\varepsilon,G_\varepsilon]| \le \frac{C_q}{|\lambda|^q} T^q
\end{align}
where we recall from \eqref{lambda,T} that $\lambda = \xi_1/|\xi_2|^{1/H}$ and $T = |\xi_2|^{1/H}$.
As mentioned in the introduction, the bound \eqref{a} will be achieved by combining Lemma \ref{lem:I[eps,G]} and Lemma \ref{lem:I:UB} below,
through an intricate integration by parts and combinatorial procedure.
In particular, it should be clear from \eqref{a} that we need $q$ integrations by parts in $\mathcal{I}_T^\lambda[\varepsilon,G]$ if we wish to achieve the bound $|\lambda|^{-q}T^q$.
Calibrating those $q$ integrations by parts in the $2q$-dimensional simplex $\Delta_T$ is of fundamental importance.

Let us set up the stage for our combinatorial formula for the integral $\mathcal{I}[\varepsilon, G_\varepsilon]$ in \eqref{2q:moment} which has been obtained in \cite{LL25} via integration by parts.
Note that whenever integration by parts is performed on a simplex like $\Delta_{T}$, one has to take care of several boundary terms. In our case, three terms are produced for each integration by parts. 
Two of those terms involve evaluation of a variable at a boundary value and hence that variable is removed.
In order to describe the formula, we need to extend the definition of $\mathcal{I}[\varepsilon, G]$ in \eqref{I} to the case that $\varepsilon = (\varepsilon_1,\dots, \varepsilon_{2q}) \in (\Z \cup\{\ast\})^{2q}$, where $\ast$ indicates the removal of an entry of $\varepsilon$. Below we give a definition of a related operation on Rademacher type variables.

\begin{definition}\label{def:Phi-j}
Let $\ep\in(\Z\cup\{\ast\})^{2q}$ with $q \ge 1$, and let $j\in \{1,\dots,2q\}$.
Suppose that $\ep_{j-1}, \ep_{j+1} \ne \ast$ if $j \not\in\{1,2q\}$;
$\ep_2\ne \ast$ if $j=1$; $\ep_{2q-1} \ne \ast$ if $j=2q$.
The operation $\Phi_j^-$ on $\ep$ is defined by
\begin{equation*}
\Phi_j^-(\varepsilon) = \begin{cases}
	(\varepsilon_1, \cdots, \varepsilon_{j-1}+\varepsilon_j, \underbrace{\ast}_{\text{$j$-th position}}, \varepsilon_{j+1}, \cdots, \varepsilon_{2q}) & \text{if $j>1$,}\\
	(\ast, \varepsilon_2, \cdots, \varepsilon_{2q}) & \text{if $j=1$.}
	\end{cases}
\end{equation*}
In the same way, the operation $\Phi_j^+$ is defined by
\begin{equation*}
\Phi_j^+(\varepsilon) = \begin{cases}
	(\varepsilon_1,\cdots,\varepsilon_{j-1},\underbrace{\ast}_{\text{$j$-th position}},\varepsilon_{j+1}+\varepsilon_j, \cdots, \varepsilon_{2q}) & \text{if $j<2q$,}\\
	(\varepsilon_1, \cdots, \varepsilon_{j-1},\ast) & \text{if $j=2q$.}
	\end{cases}
\end{equation*}
\end{definition}

\noindent
Observe that the operations $\Phi_j^-$ and $\Phi_j^+$ also define an action on a function $G:\R^{2q}\to\R^{2q}$. Namely we set 
\begin{equation}\label{Phi-G}
\Phi_j^-(G)(u_1,\cdots,u_{2q}) = \begin{cases}
	G(u_1,\cdots, u_{j-1},\underbrace{u_{j-1}}_{\text{$j$-th position}}, u_{j+1},\cdots, u_{2q}) & \text{if $j>1$,}\\
	G(0,u_2,\cdots,u_{2q}) & \text{if $j=1$,}
	\end{cases}
\end{equation}
and the action $\Phi_j^+$ is defined by
\begin{equation}\label{Phi+G}
\Phi_j^+(G)(u_1,\cdots,u_{2q}) = \begin{cases}
	G(u_1,\cdots, u_{j-1},\underbrace{u_{j+1}}_{\text{$j$-th position}}, u_{j+1},\cdots, u_{2q}) & \text{if $j<2q$,}\\
	G(u_1,\cdots,u_{2q-1},T) & \text{if $j=2q$.}
	\end{cases}
\end{equation}

Before stating our main formula for integrations by parts on simplexes, we label some more notation for indices, domains and integrals:
\begin{enumerate}[label=\textbf{(\roman*)}]
\item
Our integration by parts will involve differentiations or evaluations on functions. 
In order to perform integration by parts multiple times in a systematic way and to obtain a tractable formula, we will only integrate by parts with respect to the odd-indexed variables $u_1,u_3,\dots,u_{2q-1}$.
For this reason, we consider the following set of operations:
\begin{equation}\label{b4}
	\Sigma(q) = \Sigma_1 \times \Sigma_3 \times \cdots\times \Sigma_{2q-1}, 
	\quad \text{where}\quad
	\Sigma_j = \{\Phi_{j}^-, \Phi_{j}^+, \partial_{u_{j}}\}.
\end{equation}

\item
For each $\sigma=(\sigma_1,\sigma_3,\dots, \sigma_{2q-1}) \in \Sigma(q)$,
we may partition the indices of $\si$ as follows:
\begin{align*}
	\{1,3,\dots, 2q-1\} = J_1(\sigma) \cup J_2(\sigma) \cup J_3(\sigma),
\end{align*}
where the sets $J_{1},J_{2}, J_{3}$ are subsets of $\{1,3,\dots,2q-1\}$ defined by
\begin{equation*}
	J_1(\sigma) = \{j: \sigma_j = \Phi_{j}^-\}, \quad J_2(\sigma) = \{j: \sigma_j = \Phi_{j}^+\},\quad
	J_3(\sigma) = \{j: \sigma_j = \partial_{u_{j}}\}.
\end{equation*}
Let us also write 
\begin{align}\label{J1uJ2}
	J_1(\sigma) \cup J_2(\sigma) = \{k_1, \dots, k_m\}.
\end{align}
Note that for $j\in J_1(\sigma) \cup J_2(\sigma)$, the action $\si_{j}$ is introduced in Definition~\ref{def:Phi-j}. One can then compose the actions and form $\alpha :=\sigma_{k_m} \cdots \sigma_{k_1} (\varepsilon)$ for $\ep\in\{-1,1\}^{2q}$. We get 
\begin{align}\begin{split}\label{alpha}
	\alpha  \in \{0, \pm 1,\pm 2, \pm 3, \ast \}^{2q} 
	\quad \text{and}\quad
	\{j: \alpha_j = \ast\} = \{k_1,\dots, k_m\}.
\end{split}\end{align}
Note that all non-$\ast$ entries of $\alpha$ do not exceed $\pm3$ because we have chosen to apply operations $\sigma_j \in \Sigma_j$ only to the odd-indexed entries $j = 1,3,\dots,2q-1$.
For example, when $q=2$, the maximal value $\pm3$ may be achieved for $\alpha = \Phi_3^- \Phi_1^+(\ep) = (\ast, \ep_1+\ep_2+\ep_3, \ast, \ep_4)$.

\item
Recall that the simplex $\Delta_{T}$ is defined by~\eqref{b3}. 
For $\al$ given as in \eqref{alpha}, we let $\Delta_{\alpha, T}$ denote the simplex in $\R^{2q-m}$ with the coordinates $u_{k_1},\dots,u_{k_m}$ removed, i.e.,
\begin{align}\label{partial:simplex}
	\Delta_{\alpha,T} = \{(u_1,\dots, u_{k_1-1}, u_{k_1+1}, \dots, u_{k_m-1}, u_{k_m+1}, \dots, u_{2q}) : 0<u_1<\cdots<u_{2q}<T\}.
\end{align}

\item
For $\alpha \in (\Z \cup \{\ast\})^{2q}$, the definition of the multiple integral $\mathcal{I}[\ep, G]$ in~\eqref{I} can be extended to the partial simplex $\Delta_{\alpha, T}$ by setting
\begin{align}\label{I[alpha,G]}
	\mathcal{I}[\alpha, G] = \mathcal{I}^\lambda_T[\alpha,G] := \int_{\Delta_{\alpha,T}} e^{-2\pi i\lambda \langle\alpha, u \rangle} G(u) du,
\end{align}
which is a $2q-m$ iterated integral over $\Delta_{\alpha, T}$ and $\langle \alpha, u \rangle$ is identified as an inner product on $\R^{2q-m}$ with $\ast$ positions removed.

\end{enumerate}

\begin{example}\label{ex:b5}
In order to illustrate our integration by parts procedure, let us take a simple example of integration on an 8-dimensional simplex. The variables $\ep$ and $\sigma$ are represented in Figure~\ref{fig:f1} below.

 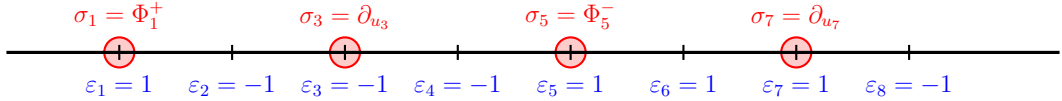
\begin{figure}[h!]
    \centering
    \begin{tikzpicture}[
    tick/.style={thick},
    circle/.style={draw=red, fill=red!20, thick}, 
    label/.style={above, font=\small}
]
\foreach \x in {1,3,5,7} {
    \draw[circle] (1.5*\x, 0) circle (0.2cm); 
}

\draw[very thick] (0,0) -- (14,0);

\foreach \x in {1,2,...,8} {
    \draw[tick] (1.5*\x, -0.1) -- (1.5*\x, 0.1);
    }
    
    
\node[below, blue, scale=.8] at (1.5, -0.2) {$\varepsilon_1 = 1$};
\node[below, blue, scale=.8] at (3, -0.2) {$\varepsilon_2 = -1$};
\node[below, blue, scale=.8] at (4.5, -0.2) {$\varepsilon_3 = -1$};
\node[below, blue, scale=.8] at (6, -0.2) {$\varepsilon_4 = -1$};
\node[below, blue, scale=.8] at (7.5, -0.2) {$\varepsilon_5 = 1$};
\node[below, blue, scale=.8] at (9, -0.2) {$\varepsilon_6 = 1$};
\node[below, blue, scale=.8] at (10.5, -0.2) {$\varepsilon_7 = 1$};
\node[below, blue, scale=.8] at (12, -0.2) {$\varepsilon_8 = -1$};

\node[above, red, scale=.8] at (1.5, 0.2) {$\si_{1}=\Phi_{1}^{+}$};
\node[above, red, scale=.8] at (4.5, 0.2) {$\si_{3}=\partial_{u_{3}}$};
\node[above, red, scale=.8] at (7.5, 0.2) {$\si_{5}=\Phi_{5}^{-}$};
\node[above, red, scale=.8] at (10.5, 0.2) {$\si_{7}=\partial_{u_{7}}$};

\end{tikzpicture}
 \caption{Example of values for the variables $\ep$ and $\si$, towards an integration in $\Delta_{\alpha,T}$}
 \label{fig:f1}
\end{figure}
\noindent
Note that for this transformation we have $J_{1}(\si)=\{5\}$, $J_{2}(\si)=\{1\}$ and $J_{3}(\si)=\{3,7\}$. In the sequel, only $\si_{1}$ and $\si_{5}$ will affect the $\ep_{j}$'s.
\end{example}

We now recap the combinatorial integration by parts formula obtained in \cite{LL25} through integrations by parts with respect to the variables $u_1, u_3,\dots,u_{2q-1}$.

\begin{lemma}\cite[Proposition 3.4]{LL25}\label{lem:I[eps,G]}
	For any $\varepsilon \in \{-1,1\}^{2q}$, let $G_\varepsilon$ be as defined in \eqref{lambda,T} and $\mathcal{I}[\varepsilon, G_\varepsilon]$ as introduced in~\eqref{I}. Also recall that $\Sigma(q)$ is the set given by~\eqref{b4}.
	Then
	\begin{align}\label{I[eps,G]}
		\mathcal{I}[\varepsilon, G_\varepsilon] = 
		\left( \frac{1}{2\pi i \lambda} \right)^q 
		\prod_{j=1}^q \frac{1}{\varepsilon_{2j-1}} \sum_{\sigma \in \Sigma(q)} 
		\mathcal{I}[\sigma; \varepsilon, G_\varepsilon].
	\end{align}
	Here, $\mathcal{I}[\sigma; \varepsilon, G_\varepsilon]$ is defined by 
	\begin{align}\label{I[sigma,G]}
		\mathcal{I}[\sigma; \varepsilon, G_\varepsilon] := (-1)^{\# J_2(\sigma)} \mathcal{I}[\sigma_{k_m}\cdots\sigma_{k_1}(\varepsilon)\,,\, \sigma_{2q-1} \cdots \sigma_3 \sigma_1 (G_\varepsilon)],
	\end{align}
	where the indices $k_1,\dots,k_m$ are given by \eqref{J1uJ2} above, $\sigma_{k_m}\cdots\sigma_{k_1}(\varepsilon)$ is defined through the actions of $\sigma_j \in \Sigma_j$ on $\ep$ introduced in Definition \ref{def:Phi-j}, $\sigma_{2q-1} \cdots \sigma_1 (G_\varepsilon)$ is defined through the actions of $\sigma_j$ on functions given by \eqref{Phi-G} and \eqref{Phi+G},
	and $\mathcal{I}[\sigma_{k_m}\cdots\sigma_{k_1}(\varepsilon)\,,\, \sigma_{2q-1} \cdots \sigma_1 (G_\varepsilon)]$ is defined according to \eqref{I[alpha,G]} above.
\end{lemma}

In order to estimate \eqref{I[eps,G]}, we need to recall more notation and properties of $\mathcal{I}[\sigma; \varepsilon, G_\varepsilon]$ defined in \eqref{I[sigma,G]} for any $\varepsilon\in \mathcal{A}_{2q}$ and $\sigma=(\sigma_1,\sigma_3,\dots, \sigma_{2q-1}) \in \Sigma(q)$.
Label and order the indices $k_1<\cdots<k_m$ and $k'_1<\cdots<k'_\ell$ such that 
\begin{align}\label{J1:J2:J3}
	J_1(\sigma)\cup J_2(\sigma) = \{k_1,\dots,k_m\}
	\quad\text{and}\quad
	J_3(\sigma) = \{k'_1,\dots,k'_\ell\},
\end{align}
where we remark that
\begin{align}\label{m+l=q}
	m+\ell=q.
\end{align}
Then $\sigma_{k_m} \cdots \sigma_{k_1} (\varepsilon)$ has $m$ entries that are equal to $\ast$ and \eqref{I[sigma,G]} can be written as
\begin{align}\label{I[sigma,G]:2}
	\mathcal{I}[\sigma; \varepsilon, G_\varepsilon] = (-1)^{\# J_2(\sigma)} \int_{\Delta_{\sigma_{k_m}\cdots\sigma_{k_1}(\varepsilon),T}} e^{-2\pi i \lambda\langle \sigma_{k_m}\cdots\sigma_{k_1}(\varepsilon), u \rangle} \sigma_{2q-1}\cdots\sigma_3\sigma_1 G_\varepsilon(u) du,
\end{align}
where the integral is a $2q-m$ iterated integral
and $\Delta_{\sigma_{k_m}\cdots\sigma_{k_1}(\varepsilon),T}$ is defined by \eqref{partial:simplex}.
We are now ready to give a simpler expression for the right hand side of \eqref{I[sigma,G]}.

\begin{lemma}\cite[Lemma 6.3]{LL25}\label{lem:variables}
	For $q \ge 1$, let $\ep \in \{-1,1\}^{2q}$ and $\sigma \in \Sigma(q)$ as defined in \eqref{b4}.
	With the notation introduced above, we set
	\begin{gather}\begin{gathered}\label{j_1-j_I}
		\{j_1<\cdots<j_I\} = \{ j : [\sigma_{k_m}\cdots\sigma_{k_1} (\varepsilon)]_j \not\in \{0, \ast\} \}\quad \text{and}\quad\\
		(a_1,\dots, a_I) = ([\sigma_{k_m}\cdots\sigma_{k_1} (\varepsilon)]_{j_1}, \dots, [\sigma_{k_m}\cdots\sigma_{k_1} (\varepsilon)]_{j_I}).
	\end{gathered}\end{gather}
	Then we have
	\begin{align}\label{E:sigmaG}
	\sigma_{2q-1} \cdots \sigma_3 \sigma_1 G_\varepsilon(u) = \partial_{u_{k_1'}}\cdots \partial_{u_{k_\ell'}}G_{\sigma_{k_m} \cdots \sigma_{k_1} (\varepsilon)}(u),
\end{align}
	which is a function depending on the variables $u_{j_1},\dots, u_{j_I}$ and independent of all other variables $u_j$ with $j \not\in\{j_1, \dots, j_I\}$.
	If we label the non-trivial variables $u_{j_1},\dots,u_{j_I}$ as $s_1,\dots,s_I$,
	then all differentiation variables $u_{k_1'},\dots,u_{k_\ell'}$ are non-trivial variables and can be labelled as $s_{p_1},\dots,s_{p_\ell}$ for a subset of indices $p_1<\cdots<p_\ell$.
	Hence, one can recast \eqref{E:sigmaG} as
	\begin{align}\label{E:sigmaG:s}
		\sigma_{2q-1}\cdots\sigma_3\sigma_1 G_\ep(u) = \partial_{s_{p_1}}\cdots\partial_{s_{p_\ell}} G_{\sigma_{k_m} \cdots \sigma_{k_1} (\varepsilon)}(s).
	\end{align}
\end{lemma}

\smallskip

\begin{example}\label{ex:b6}
Let us go back to Example~\ref{ex:b5}. The transformation of the variables $\ep$ through
the action of $\si$ given in Figure~\ref{fig:f1} is summarized in Figure~\ref{fig:f2} below.
 \begin{figure}[h!]
    \centering
    \begin{tikzpicture}[
    tick/.style={thick},
    circle/.style={draw=red, fill=red!20, thick}, 
    label/.style={above, font=\small}
]
\foreach \x in {1,3,5,7} {
    \draw[circle] (1.5*\x, 0) circle (0.2cm); 
}

\draw[very thick] (0,0) -- (14,0);

\foreach \x in {1,2,...,8} {
    \draw[tick] (1.5*\x, -0.1) -- (1.5*\x, 0.1);
    }
    
    
\node[below, blue, scale=.8] at (1.5, -0.25) {$\alpha_1 = *$};
\node[below, blue, scale=.8] at (3, -0.2) {$\alpha_2 = 0$};
\node[below, blue, scale=.8] at (4.5, -0.2) {$\alpha_3 = -1$};
\node[below, blue, scale=.8] at (6, -0.2) {$\alpha_4 = 0$};
\node[below, blue, scale=.8] at (7.5, -0.25) {$\alpha_5 = *$};
\node[below, blue, scale=.8] at (9, -0.2) {$\alpha_6 = 1$};
\node[below, blue, scale=.8] at (10.5, -0.2) {$\alpha_7 = 1$};
\node[below, blue, scale=.8] at (12, -0.2) {$\alpha_8 = -1$};

\node[above,  dg, scale=.8] at (4.5, 0.3) {$a_{1}=-1$};
\node[right,  scale=.8] at (4.7, 0.2) {${\cre \partial_{a_{1}}}$};
\node[above, dg, scale=.8] at (9, 0.3) {$a_{2}=1$};
\node[above, dg, scale=.8] at (10.5, 0.3) {$a_{3}=1$};
\node[right,  red, scale=.8] at (10.7, 0.2) {$\partial_{a_{3}}$};
\node[above, dg, scale=.8] at (12, 0.3) {$a_{4}=-1$};

\end{tikzpicture}
 \caption{Example of action in the 6-dimensional partial simplex $\Delta_{\alpha,T}$ where $\alpha = \si_{1}\si_{5}(\ep)$.
 }
 \label{fig:f2}
\end{figure}
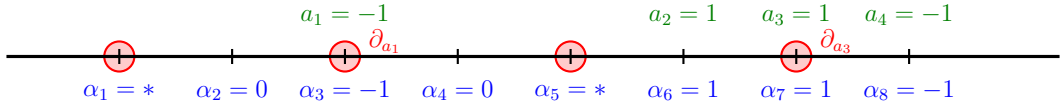
\noindent
Note that in this case the non-trivial integration variables (related to the non-trivial variables $a$) are $s_{1},\ldots,s_{4}$.
In addition, the differentiation indices  are $p_{1}=1$ and $p_{2}=3$. Therefore the collection $\mathscr{P}_2\{p_1,\dots,p_\ell\}$ introduced below in Lemma~\ref{lem:fdb} is reduced to
\begin{equation}\label{b7}
\mathscr{P}_2\{1,3\}
=\Big\{  \big[ \{1\} , \, \{3\} \big] , \,   \big[ \{1 , 3\} \big] \Big\} .
\end{equation}
\end{example}

Our next lemma labels a property of the indices $p$ defined above.
It will be useful in our computations below.

\begin{lemma}\cite[Lemma 6.6]{LL25}\label{lem:p_i}
	Let $p_1,\dots,p_\ell$ be the indices introduced in Lemma \ref{lem:variables}. 
	Then we have $p_{i+1}-p_i \ge 2$ for $1\le i < \ell$.
\end{lemma}

We now focus on the computation of the derivatives on the right hand side of \eqref{E:sigmaG:s}.
To this aim, recall the following derivative formula obtained in \cite{LL25} using Fa\`a di Bruno's formula.

\begin{lemma}\cite[Proposition 6.9]{LL25}\label{lem:fdb}
	Consider $\varepsilon \in \mathcal{A}_{2q}$ and the function $G_\ep$ defined in \eqref{lambda,T}.
	For $\sigma \in \Sigma(q)$, recall that the non-trivial entries of $\sigma_{k_m}\cdots\sigma_{k_1}(\ep)$ are written as $a = (a_1,\dots,a_I)$ in \eqref{j_1-j_I}.
	Let $s=(s_1,\dots,s_I)$ be the variables introduced in Lemma \ref{lem:variables}.
	Then, the derivatives in equation \eqref{E:sigmaG:s} can be expressed as
	\begin{align*}
		\partial_{s_{p_1}}\cdots \partial_{s_{p_\ell}}G_{\sigma_{k_m} \cdots \sigma_{k_1} (\varepsilon)}(s) = \sum_{P \in \mathscr{P}_2\{p_1,\dots,p_\ell\}} \pi^{\# P} \exp(\pi g_a(s)) \prod_{B \in P} g_a^{(B)}(s),
	\end{align*}
	where $\mathscr{P}_2\{p_1,\dots, p_\ell\}$ denotes the set of all partitions $P$ of $\{p_1,\dots, p_\ell\}$ such that $\# B \le 2$ for all $B \in P$,
	\begin{align}\label{g_a}
		g_a(s) = -\mathrm{Var}\left( \sum_{i=1}^I a_i B_{s_i} \right),
	\end{align}
	and $g_a^{(B)}(s) = \partial_{s_{i_1}}\cdots \partial_{s_{i_n}} g_a(s)$ if $\# B = n$ and $B = \{s_{i_1},\dots,s_{i_n}\}$.
\end{lemma}

As a result of the previous considerations, we obtain the following bound for $\mathcal{I}[\sigma; \varepsilon, G_\varepsilon]$ defined in \eqref{I[sigma,G]}.

\begin{proposition}\label{pr:I[sigma,G]}
For any $\varepsilon \in \{-1,1\}^{2q}$ and $\sigma \in \Sigma(q)$ as defined in \eqref{b4}, let $G_\varepsilon$ be the function defined in \eqref{lambda,T} and $\mathcal{I}[\sigma; \varepsilon, G_\varepsilon]$ be as defined in \eqref{I[sigma,G]}.
Recall the indices $k_1,\dots,k_m$ and $k_1',\dots,k_\ell'$ defined in \eqref{J1:J2:J3}.
Also recall the variables $s_1,\dots,s_I$ and indices $p_1,\dots,p_\ell$ introduced in Lemma \ref{lem:variables}.
Then we have
\begin{align*}
	|\mathcal{I}[\sigma; \varepsilon, G_\varepsilon]|
	\le T^{2q-m-I} \sum_{P \in \mathscr{P}_2\{p_1,\dots, p_\ell\}} \pi^{\# P} \int_{0<s_1<\cdots<s_I<T} \exp(\pi g_a(s)) \prod_{B \in P}|g_a^{(B)}(s)| ds
\end{align*}
where $g_a$ is defined in \eqref{g_a} and $a=(a_1,\dots,a_I)$ is defined in \eqref{j_1-j_I}.
\end{proposition}

\begin{proof}
Thanks to \eqref{I[sigma,G]:2} and Lemma \ref{lem:variables}, we have
\begin{align}\notag
	|\mathcal{I}[\sigma; \varepsilon, G_\varepsilon]| 
	&= \left| \int_{\Delta_{\sigma_{k_m}\cdots\sigma_{k_1}(\varepsilon),T}} e^{-2\pi i \lambda\langle \sigma_{k_m}\cdots\sigma_{k_1}(\varepsilon), u \rangle} \partial_{u_{k_1'}}\cdots \partial_{u_{k_\ell'}}G_{\sigma_{k_m} \cdots \sigma_{k_1} (\varepsilon)}(u)du\right|\\
	& \le \int_{\Delta_{\sigma_{k_m}\cdots\sigma_{k_1}(\varepsilon),T}} \big| \partial_{u_{k_1'}}\cdots \partial_{u_{k_\ell'}}G_{\sigma_{k_m} \cdots \sigma_{k_1} (\varepsilon)}(u) \big| du,
	\label{I[sigma,G]<int}
\end{align}
which involves a $2q-m$ iterated integral on the partial simplex $\Delta_{\sigma_{k_m}\cdots\sigma_{k_1}(\varepsilon),T}$ with $m$ coordinates removed, as defined in \eqref{partial:simplex}.
According to Lemma \ref{lem:variables}, the integrand in \eqref{I[sigma,G]<int} depends only on $I$ variables $u_{j_1},\dots,u_{j_I}$, which are relabelled as $s_1,\dots,s_I$ (see \eqref{j_1-j_I} and \eqref{E:sigmaG:s} above), and does not depend on the remaining $2q-m-I$ variables.
Integrating out those $2q-m-I$ variables and keeping in mind that we are integrating on a simplex, we obtain
\begin{align*}
	|\mathcal{I}[\sigma; \varepsilon, G_\varepsilon]| 
	\le T^{2q-m-I} \int_{0<s_1<\cdots<s_I<T} \big| \partial_{s_{p_1}}\cdots \partial_{s_{p_\ell}}G_{\sigma_{k_m} \cdots \sigma_{k_1} (\varepsilon)}(s) \big| ds.
\end{align*}
We can then apply the Fa\`{a} di Bruno formula in Lemma \ref{lem:fdb} to finish the proof.
\end{proof}

\subsection{Estimates for the variance function and its derivatives}
In the previous section, we have reduced our main estimate to a multiple integral involving the function $g_a$ in \eqref{g_a} and its derivatives.
In this section we gather some basic facts about the singularity of those functions.
We start by recalling the expression of some derivatives in the lemma below.

\begin{lemma}\cite[Lemma 6.7]{LL25}
	Let $H \in (0,1)$.
	Let $0<s_1<\cdots<s_I$ and $a = (a_1,\dots,a_I) \in \R^I$ with $\sum_{i=1}^I a_i = 0$. Then:
	\begin{enumerate}
	\item[(i)] For any $i \in \{1,\dots, I\}$,
	\begin{align}\label{E:dg}
		\partial_{s_i} g_a(s) = 2H \sum_{j=1}^{i-1} a_i a_j (s_i-s_j)^{2H-1} - 2H \sum_{j=i+1}^I a_i a_j (s_j-s_i)^{2H-1}.
	\end{align}
	\item[(ii)] For any $i<j$ in $\{1,\dots, I\}$,
	\begin{align}\label{E:ddg}
		\partial_{s_i} \partial_{s_j} g_a(s) = 2H(1-2H) a_i a_j (s_j-s_i)^{2H-2}.
	\end{align}
	\item[(iii)] If $k \ge 3$, then for all $i_1<\cdots<i_k$ in $\{1,\dots, I\}$,
	\begin{align}\label{E:d...dg}
		\partial_{s_{i_1}} \cdots \partial_{s_{i_k}} g_a(s) = 0.
	\end{align}
	\end{enumerate}
\end{lemma}

\begin{remark}
Recall that we have chosen to apply operations $\sigma_j \in \Sigma_j$ only for the odd indices $j = 1, 3, \dots, 2q-1$.
It follows that the variables being differentiated in \eqref{E:sigmaG} must be a subset of $u_1, u_3, \dots, u_{2q-1}$ and each of those variables is differentiated at most once, so the corresponding differentiation variables $s_{p_1}, \dots, s_{p_\ell}$ introduced in Lemma \ref{lem:variables} are all non-repeating. 
Therefore, the mixed derivatives in \eqref{E:d...dg} of order 3 or higher are all vanishing and this is why in the Fa\`{a} di Bruno formula in Lemma \ref{lem:fdb} we only have terms that involve $B \in P$ with the restriction $\# B \le 2$.
\end{remark}


With the expressions for $\partial_{s_{i_1}} \cdots \partial_{s_{i_k}} g_a$ in hand, we now turn to some estimates for the derivatives of $g_a$ when $H<1/2$.
Note that the estimates are different from those in \cite{LL25} where $H>1/2$.

\begin{lemma}\label{lem:dg}
	Let $H\in (0,1/2)$. Let $0<s_1<\cdots<s_I$ and $a = (a_1,\dots,a_I) \in \R^I$ with $\sum_{i=1}^I a_i = 0$. 
	Recall that the function $g_a$ is defined by \eqref{g_a}.
	Then:
	\begin{enumerate}
	\item 
	For any $i \in \{1,\dots, I\}$,
	\begin{align}\label{dg}
		|\partial_{s_i}g_a(s)| \le \|a\|_\infty^2\, I \left[ (s_i-s_{i-1})^{2H-1}{\bf 1}_{\{i \ge 2\}} + (s_{i+1}-s_i)^{2H-1} {\bf 1}_{\{i \le I-1\}} \right].
\end{align}
	\item 
	For any $i<j$ in $\{1,\dots, I\}$,
	\begin{align}\label{ddg}
		|\partial_{s_i} \partial_{s_j} g_a(s)| \le \|a\|_\infty^2 (s_j-s_i)^{2H-2}.
\end{align}
	\end{enumerate}
\end{lemma}

\begin{proof}
For $i = 1$ or $I$, we may use \eqref{E:dg} and monotonicity (since $2H-1<0$) to deduce that
\begin{align*}
	|\partial_{s_i}g_a(s)|\le \begin{cases}
	 \|a\|_\infty^2 \sum_{j=2}^I  (s_j-s_1)^{2H-1}
	\le  \|a\|_\infty^2\, I  (s_2-s_1)^{2H-1} & \text{if $i=1$,}\\
	\|a\|_\infty^2 \sum_{j=1}^{I-1} (s_I-s_j)^{2H-1} \le \|a\|_\infty^2 \, I (s_I-s_{I-1})^{2H-1} & \text{if $i=I$.}
	\end{cases}
\end{align*}
Similarly, for $2\le i \le I-1$,
\begin{align*}
	|\partial_{s_i}g_a(s)| &\le \|a\|_\infty^2 \sum_{j=1}^{i-1} (s_i-s_j)^{2H-1} + \|a\|_\infty^2 \sum_{j=i+1}^I  (s_j-s_i)^{2H-1}\\
	& \le \|a\|_\infty^2\, I (s_i-s_{i-1})^{2H-1} + \|a\|_\infty^2\, I  (s_{i+1}-s_i)^{2H-1}.
\end{align*}
This proves \eqref{dg}. 
Finally, \eqref{ddg} follows immediately from \eqref{E:ddg} and $H \in (0,1/2)$.
\end{proof}

Let us also recall a lower bound on the variance function obtained in \cite[Proposition 7.6]{LL25} using the strong local nondeterminism property of fractional Brownian motion.

\begin{lemma}\cite[Proposition 7.6]{LL25}\label{lem:var}
Let $H \in (0,1)$ and consider a fractional Brownian motion $B$ with Hurst parameter $H$. There exists a constant $C_H>0$ depending only on $H$ such that for all integers $n \ge 1$, for all $0<t_1<\cdots<t_n$, and for all $a_1, \dots, a_n \in \R$
\begin{align*}
	\mathrm{Var}\left( \sum_{j=1}^n a_j B_{t_j}\right) \ge \frac{C_H^{n-1}}{n} \sum_{j=1}^n (a_j+\cdots+a_n)^2 (t_j-t_{j-1})^{2H},
\end{align*}
where we have used the convention $t_0=0$.
\end{lemma}

\section{Proof of Theorem \ref{th:main}}

Recall the quantity $\mathcal{I}[\sigma; \varepsilon, G_\varepsilon]$ defined in \eqref{I[sigma,G]}.
The proof of our main theorem relies on the following key estimate.

\begin{lemma}\label{lem:I:UB}
There exists $C < \infty$ depending only on $H$ such that
\begin{align*}
	|\mathcal{I}[\sigma; \varepsilon, G_\varepsilon]| \le C^{q^2} T^q
\end{align*}
uniformly for all $q \in \N_+$, for all $\sigma \in \Sigma(q)$, for all $\varepsilon \in \mathcal{A}_{2q}$, for all $T>0$, and for all $\lambda >0$.
\end{lemma}

\begin{proof}
Fix $q, \sigma, \varepsilon, T, \lambda$ as in the statement of the lemma.
Recall from Proposition \ref{pr:I[sigma,G]} that we can bound \eqref{I[sigma,G]} as follows:
\begin{align}\label{I[sigma,G]:bd}
	|\mathcal{I}[\sigma; \varepsilon, G_\varepsilon]|
	\le T^{2q-m-I} \sum_{P \in \mathscr{P}_2\{p_1,\dots, p_\ell\}} \pi^{\# P} \int_{0<s_1<\cdots<s_I<T} \exp(\pi g_a(s)) \prod_{B \in P}|g_a^{(B)}(s)| ds.
\end{align}
We now handle the right hand side of \eqref{I[sigma,G]:bd}, dividing our proof in several steps.

\noindent
{\it Step 1: Reduction to a singularity analysis.}
Recall that $\ep$ is an element of the set $\mathcal{A}_{2q}$ defined by~\eqref{A}, which means that $\sum_{j=1}^{2q}\ep_j = 0$.
This null average property is easily seen to be preserved by the operations $\Phi^\mp_j$ introduced in Definition \ref{def:Phi-j}.
Therefore, we also have $\sum_{j=1}^I a_j = 0$, where $(a_1,\dots,a_I)$ is defined by \eqref{j_1-j_I}.
Starting from this elementary observation, we define $1=i_1<\cdots<i_d < I$ such that
\begin{align}\label{J*}
	a_i+a_{i+1}+\cdots+a_I \begin{cases}
	=0 & \text{for $i\in \{i_1,\dots,i_d\}$,}\\
	\ne 0 & \text{for $i \in J^* := \{1,\dots, I\}\setminus\{i_1,\dots,i_d\}$.}
	\end{cases}
\end{align}
By Lemma \ref{lem:var} combined with \eqref{J*}, only the indices in $J^{*}$ will contribute to a proper upper bound for the exponential terms in $\mathcal{I}[\sigma; \varepsilon, G_\varepsilon]$. Specifically,
there exists a constant $0<K<1$ depending only on $H$ such that
\begin{align}\label{SLND}
	\exp(\pi g_a(s)) = \exp\left\{-\pi \mathrm{Var}\left( \sum_{i=1}^I a_i B(s_i) \right)\right\}
	\le \exp\left\{ - K^q \sum_{i \in J^*} (s_i-s_{i-1})^{2H} \right\}
\end{align}
uniformly for all $\varepsilon \in \mathcal{A}_{2q}$, $\sigma \in \Sigma(q)$ and $0<s_1<\cdots<s_I<T$.
Then, applying \eqref{SLND} and Lemma \ref{lem:dg} to \eqref{I[sigma,G]:bd} yields
\begin{align}\begin{split}\label{I:UB}
	&|\mathcal{I}[\sigma; \varepsilon, G_\varepsilon]|\\
	&\le (18I\pi)^q T^{2q-m-I} \sum_{P \in \mathscr{P}_2} \sum_{(\theta_1,\dots, \theta_I) \in \Theta(P)} \int_{0<s_1<\cdots<s_I<T} \exp\left\{- K^q \sum_{i \in J^*} (s_i-s_{i-1})^{2H}\right\} \\
	& \qquad \quad \times \prod_{B \in P: \# B = 2, B = \{j,i\}, j<i} (s_i-s_{j})^{2H-2} 
	\prod_{B \in P : \# B = 1, B = \{i\}} |s_{i}-s_{i +\theta_i}|^{2H-1}ds,
\end{split}\end{align}
where $\mathscr{P}_2 = \mathscr{P}_2\{p_1,\dots, p_\ell\}$ is introduced in Lemma \ref{lem:fdb}
and $\Theta(P)$ describes the set of indices affected by the differentiation rule \eqref{dg}, namely,
\begin{align}\begin{split}\label{Theta(P)}
	\Theta(P) = \big\{ (\theta_1,\dots,\theta_I) \in \{-1,0,1\}^I : \ &\theta_i = 0 \text{ if } \{i\} \not\in P;\\
	&\theta_i \in \{-1,1\} \text{ if } \{i\} \in P \setminus \{1, I\}\,; 
	\\
	&\theta_1 = 1 \text{ if } \{1\} \in P\,; \theta_I = -1 \text{ if } \{I\} \in P \big\}.
\end{split}\end{align}
From \eqref{I:UB} we see that the value of $\theta_i$ is not important if $\{i\} \not\in P$; in this case, we have set $\theta_i = 0$ for convenience.
In \eqref{I:UB}, the factor $\pi^q$ comes from the factor $\pi^{\# P}$ in \eqref{I[sigma,G]:bd} and the elementary bound $\# P \le \ell \le q$ (see \eqref{m+l=q} above); the factor $(18I)^q$ comes from the derivative bounds in Lemma \ref{lem:dg}, the bound $\|a\|_\infty \le 3$ (see \eqref{alpha} and \eqref{j_1-j_I} above), together with the fact that the number of factors in the product $\prod_{B \in P} (\cdots)$ in \eqref{I[sigma,G]:bd} is equal to $\# P$ which is again at most $q$.

In order to handle the right hand side of \eqref{I:UB}, we need some extra notation.
That is, we define $U = U(I,P,J^*,\theta)$ as the integral
\begin{align}\label{D:U}
	U(I, P, J^*, \theta)&=\int_{0<s_1<\cdots<s_I<T} \exp\left\{- K^q \sum_{i \in J^*} (s_i-s_{i-1})^{2H}\right\}\\ 
	&\quad \times \prod_{B \in P: \# B = 2, B = \{j,i\}, j<i} (s_i-s_{j})^{2H-2}
	\prod_{B \in P : \# B = 1, B = \{i\}} |s_{i}-s_{i +\theta_i}|^{2H-1}ds.\notag
\end{align}
Note that $U$ depends on the number of active variables $I$, plus the following parameters:
\begin{enumerate}[label=\textbf{(\alph*)}]
\item $P$ is a generic collection of disjoint subsets of $\{ 1, \dots, I\}$ such that $\# B = 1$ or $2$ for all $B \in P$ and $\sum_{B \in P} \# B = \ell$;
\item $J^*$ is a subset of $\{1,\dots, I\}$;
\item $\theta = (\theta_1,\dots,\theta_I)$ is a tuple in the set $\Theta(P)$ introduced in \eqref{Theta(P)}.
\end{enumerate}
 With this notation in hand, one can recast relation~\eqref{I:UB} as
\begin{equation}\label{c1}
|\mathcal{I}[\sigma; \varepsilon, G_\varepsilon]|
	\le 
	(18I\pi)^q T^{2q-m-I} \sum_{P \in \mathscr{P}_2} \sum_{(\theta_1,\dots, \theta_I) \in \Theta(P)} 
	U(I,P,J^*,\theta) .
\end{equation}
In the sequel, our goal is to show that
\begin{align}\label{U:UB}
	U(I, P, J^*, \theta) \le C^{q} T^{I-\ell}
\end{align}
for some $C<\infty$ depending only on $H$. The proof of this relation is detailed below.

\smallskip

\begin{example}\label{ex:I:bd}
In the situation described in Examples \ref{ex:b5} and \ref{ex:b6}, we have seen that the sum in \eqref{I[sigma,G]:bd} is a sum over $\mathscr{P}_2\{1,3\}$ (this set being given by \eqref{b7}). Next, referring to the variables $a$ in Figure \ref{fig:f2}, the set $J^* \subset \{1,2,3,4\}$ is $J^* = \{2,4\}$.
Moreover, recalling that for $P \in \mathscr{P}_2\{1,3\}$, $\Theta(P)$ is given by \eqref{Theta(P)} and $\theta_i = 0$ if $\{i\} \not\in P$, we have
\begin{align*}
	\Theta\left( \big[ \{1,3\} \big] \right) = \{(0,0,0,0)\} \, ,
	\quad \text{and}\quad
	\Theta\left( \big[ \{1\}, \{3\} \big] \right) = \{ (1,0,\theta_3,0) : \theta_3 = \pm 1 \}.
\end{align*}
Therefore, relation \eqref{c1} reads
\begin{align}\begin{split}\label{I<U+U+U}
	|\mathcal{I}[\sigma; \ep, G_\ep]| 
	&\le (18 I \pi)^9 T^{2q-m-I}\\
	&\quad\times\Big[U\left(I=4, P = \big[\{1,3\}\big], J^* = \{2,4\}, \theta = (0,0,0,0)\right)\\
	&\qquad+U\left(I=4, P=\big[\{1\},\{3\}\big], J^* = \{2,4\}, \theta = (1,0,1,0)\right)\\
	&\qquad+U\left(I=4, P=\big[\{1\},\{3\}\big], J^*=\{2,4\}, \theta = (1,0,-1,0) \right)\Big].
\end{split}\end{align}
For the sake of clarity, let us also spell out one of the $U$ terms in \eqref{I<U+U+U}.
That is, we have
\begin{align}\begin{split}\label{U:ex}
	&U\left(I=4, P = \big[\{1,3\}\big], J^* = \{2,4\}, \theta = (0,0,0,0)\right)\\
	&= \int_{0<s_1<\cdots<s_4 <T} \exp\left\{ - K^4\left[ (s_2-s_1)^{2H}+(s_4-s_3)^{2H} \right] \right\} (s_3-s_1)^{2H-2} ds_1 \cdots ds_4,
\end{split}\end{align}
and we let the reader check how to write the other terms in \eqref{I<U+U+U}.
\end{example}

\noindent
{\it Step 2: Further remarks on our transformations $\sigma$:}
Before getting into the bulk of our estimates,
let us remind ourselves some important properties:
\begin{enumerate}[label=\textbf{(\roman*)}]

\item\label{it:sigma-i}
By Lemma \ref{lem:variables}, if $\sigma_j = \partial_{u_{j}}$ (where $j$ is odd), then the ${j}^{\text{th}}$ entry of $\sigma_{k_m}\cdots\sigma_{k_1}(\varepsilon)$ is non-trivial and $u_{j}$ corresponds to some non-trivial variable $s_i$.
In this case, we use the notation $\tau(s_i)=\partial$ to indicate that $s_i$ is a differentiation variable, and use $\tau(s_i)=0$ to indicate that $s_i$ is not a differentiation variable.
Note that $P = \{ p \in \{1,\dots,I\} : \tau(s_p) = \partial \}$.

\item\label{it:sigma-ii} 
The differentiation variables are indexed by $s_{p_1},\dots, s_{p_\ell}$. By Lemma \ref{lem:p_i}, $p_{i+1}-p_i \ge 2$ for all $1 \le i < \ell$.

\item\label{it:sigma-iii} 
The last variable $s_I$ is always not a differentiation variable. To see this, let $i$ be the largest index for which $\sigma_{2i-1} = \partial_{u_{2i-1}}$ (if there is no such index, the asserted property
clearly holds). Then the last $q-i$ operations $\sigma_{2j-1}$ ($i+1 \le j \le q$) are $\Phi^\pm_{2j-1}$ operations, which may only affect the last $2q-(2i-1) = 2q-2i+1$ entries of $\varepsilon$, and each $\Phi_{2j-1}^\pm$ operation can trivialize at most two entries of $\varepsilon$ (since the $2j-1^{\text{th}}$ entry becomes $\ast$ and the $2j-1\pm1^{\text{th}}$ entry may or may not become 0).
This means that those $q-i$ operations together can trivialize at most $2q-2i$ entries among the last $2q-2i+1$ entries, so the last non-trivial entry must be found among those entries and corresponds to a non-differentiation variable.
Therefore, $s_I$ is always not a differentiation variable. 

\item\label{it:sigma-iv} 
Recall that $a=(a_1,\dots,a_I)$ denotes the non-zero, non-$\ast$ entries of $\sigma_{k_m}\cdots\sigma_{k_1}(\varepsilon)$. If we had 
\begin{equation*}
a_{i}+a_{i+1}+\cdots+a_{I}=0 ,
\quad\text{and}\quad
a_{i+1}+\cdots+a_{I}=0 \, ,
\end{equation*}
this would trivially imply $a_{i}=0$, which is impossible. Therefore we cannot have $a_{i}+a_{i+1}+\cdots+a_{I}=0$ for two consecutive indices $i$ and $i+1$. This means that for any $i$ we always have either $i \in J^*$ or $i+1 \in J^*$ (or both $i,i+1\in J^{*}$). 
\end{enumerate}

\smallskip

\noindent
{\it Step 3: Estimate for small $T$}.
Consider $0 \le T \le 1$. In this case, we may bound the exponential factor in \eqref{D:U} by 1 to get that
\begin{align}\label{c2}
	U(I, P, J^*, \theta) &\le \int_{0<s_1<\cdots<s_I<T}  \prod_{B=\{p',p\} \in P, \,p'<p} (s_{p}-s_{p'})^{2H-2} 
	\prod_{B=\{p_i\} \in P} |s_{p_i}-s_{p_i+ \theta_i}|^{2H-1}ds.
\end{align}
Next, we integrate $s_I$ first, then $s_{I-1}$, and so on. Furthermore, for a
fixed $P\in \mathscr{P}_2 = \mathscr{P}_2\{p_1,\dots, p_\ell\}$, we define the parameters
\begin{align*}
	&r = r(P):= \# \{ B \in P : \# B = 2\},\\
	&n=n(P) := \# \{ B \in P : \# B = 1 \}=\ell-2r.
\end{align*}
With this notation in hand, the integrals in~\eqref{c2} can be estimated thanks to the following facts:
\begin{enumerate}[label=$\bullet$]
\item If $s_j = s_{p_i}$ with $\{p_i\} \in P$, then the integral can be bounded by a factor of $(2H)^{-1}T^{2H}$ when integrating $s_{p_i}$ (or $s_{p_i+1}$ respectively) if $\theta_i = -1$ (if $\theta_i = 1$ respectively), and there are $n$ such factors.
\item If $s_j = s_p$ with $B = \{p', p\} \in P$ for some $p'<p$, then the singularity $(s_{p}-s_{p'})^{2H-2}$ seems to be a non-trivial obstacle to integrability when $H \in (0,1/2)$. However, recall that for such a $B$ we have seen that
$p' \le p-2$ and $s_{p'+1}$ is not a differentiation variable (by property \ref{it:sigma-ii} above). This means that the diagonal is avoided here, and we have
\begin{align}\label{c2.1}
	\int_{s_{p-1}}^T (s_p-s_{p'})^{2H-2} ds_p \le \frac{1}{1-2H} (s_{p-1}-s_{p'})^{2H-1} \le \frac{1}{1-2H} (s_{p'+1}-s_{p'})^{2H-1}.
\end{align}
Furthermore later, when it comes to integrating $s_{p'+1}$, the factor $(s_{p'+1}-s_{p'})^{2H-1}$ from \eqref{c2.1} will be the only factor that involves $s_{p'+1}$, which is not a differentiation variable, so a simple factor of $(2H)^{-1}(1-2H)^{-1}T^{2H}$ is produced. There are $r$ such factors.
\item If $s_j$ is none of the above cases, 
then a factor of $T$ is produced, and there are $I-\ell$ such factors.
\end{enumerate}
With those considerations in mind, plugging the above estimates into~\eqref{c2}, for $0\le T \le 1$ we have obtained
\begin{align}\label{c3}
	U(I, P, J^*, \theta)
	\le C^{n+r} T^{I-\ell} T^{2H(n+r)}\le C^{q} T^{I-\ell} \, ,
\end{align}
for some $C<\infty$ that depends only on $H$. Otherwise stated, we have proved~\eqref{U:UB} for $0\le T \le 1$.

\smallskip

\noindent
{\it Step 4: Large $T$ estimate, initiating the recursion}.
It remains to show that~\eqref{U:UB} holds for $T \ge 1$.
We estimate $U(I, P, J^*, \theta)$ by first integrating the variable $s_I$, taking now the exponential factor into account.
To this aim, recall the type $\tau(s_i)$ of variables defined in property~\ref{it:sigma-i}, and also recall from property~\ref{it:sigma-ii} that $\tau(s_I) = 0$.
Next from \eqref{J*}, we see that $I \in J^*$ since $a_I \ne 0$. Hence summarizing our considerations on the $s_{I}$ variable, we are left with the following two cases:
\begin{enumerate}[label=(\alph*)]
\item 
If $\tau(s_{I-1}) = \partial$, $\{I-1\} \in P$, and $\theta_{I-1} = 1$, then the variable $s_{I}$ appears in one differentiation in~\eqref{D:U}. 
Therefore we use the bound
\begin{equation}\label{c4}
\int_{s_{I-1}}^T (s_I-s_{I-1})^{2H-1}  e^{-K^q(s_I-s_{I-1})^{2H}} ds_I \le \int_0^\infty s_I^{2H-1}  e^{-K^qs_I^{2H}} ds_I \le C^q;
\end{equation}

\item
In all other situations, the variable  $s_{I}$ is not differentiated. Hence we simply use
\begin{equation}\label{c5}
	\int_{s_{I-1}}^T  e^{-K^q(s_I-s_{I-1})^{2H}} ds_I \le \int_0^\infty  e^{-K^q s_I^{2H}} ds_I \le C^q.
\end{equation}
\end{enumerate}

\noindent
Putting together~\eqref{c4} and~\eqref{c5}, 
it follows that the computation of $U(I, P, J^*, \theta)$ is reduced to an integration on an $(I-1)$-dimensional simplex:
\begin{align}\label{U:I_0}
	U(I,P,J^*,\theta) \le C^q U(I_0, P_0,J_0^*, \theta^{(0)}),
\end{align}
where $C<\infty$ is a constant depending only on $H$ and the new parameters $I_0, P_0,J_0^*, \theta^{(0)}$ are defined as
\begin{align}\begin{split}\label{para:k=0}
&I_0 = I-1, \quad	P_0 = \{ B \in P : B \subset \{1,\dots, I-1\} \}, \quad J^*_0 = \{ j \in J^* : j \le I-1 \},\\
&\text{and} \quad \theta^{(0)} = \begin{cases}
(\theta_1,\dots,\theta_{I-2},0) & \text{if $\theta_{I-1}=1$,}\\
(\theta_1,\dots,\theta_{I-2},\theta_{I-1}) & \text{if $\theta_{I-1} \ne 1$.}
\end{cases}
\end{split}\end{align}
Note that the last entry of $\theta^{(0)}$ is set to be 0 if $\theta_{I-1}=1$ because the factor $(s_I-s_{I-1})^{2H-1}$ is removed after the integration in \eqref{c4}. In particular, $\theta^{(0)}_{I-1}$ is either $-1$ or $0$. 

Next, we continue with the estimate iteratively by integrating two variables at a time and updating $(I_0,P_0,J_0^*, \theta^{(0)})$ after each iteration. Define $\nu$, the number of iterations, by
\begin{align}\label{nu}
	\nu = \begin{cases}
	I_0/2 & \text{if $I_0$ is even;}\\
	(I_0-1)/2 & \text{if $I_0$ is odd.}
	\end{cases}
\end{align}
For any $k \in \{1,\dots, \nu-1\}$, given $(I_{k-1}, P_{k-1}, J_{k-1}^*, \theta^{(k-1)})$, we set $i = I_{k-1}$ and we update the set of parameters by defining
\begin{equation}\label{c6}
I_k = I_{k-1}-2 = i-2,
\quad
L_k=\{ j \in \{i-1, i\} :  \tau(s_j) = \partial \},
\quad 
\ell_k=\# L_k, 
\end{equation}
as well as
\begin{equation*}
	P_k=\{ B \setminus L_k : B \in P_{k-1} \},
	\quad\text{and}\quad
	J^*_k=\{ j \in J^* : j \le I_k \}.
\end{equation*}
Note that $P_k$ is a partition of $P \cap \{1,\dots,I_k\}$.
We also update the parameter $\theta$ by setting
\begin{align}\label{theta:update}
	\theta^{(k)} &= \begin{cases}
	\big(\theta^{(k-1)}_1,\dots,\theta^{(k-1)}_{i-3}, 0\big) & \text{if $\theta^{(k-1)}_{i-2} = 1$;}\smallskip\\
	\big(\theta^{(k-1)}_1,\dots,\theta^{(k-1)}_{i-3}, \theta^{(k-1)}_{i-2}\big) + {\bf e}_j& \text{if $\{j, i-1\} \in P_{k-1}$ for some $j<i-2$;}\smallskip\\
	\big(\theta^{(k-1)}_1,\dots,\theta^{(k-1)}_{i-3}, \theta^{(k-1)}_{i-2}\big) & \text{otherwise}
	\end{cases}
\end{align}
where ${\bf e}_j$ denotes the vector $(0,\dots, 1,\dots, 0)$ whose $j^{\text{th}}$ entry is 1, and 0 otherwise.
Observe that the last entry of $\theta^{(k)}$ is either $-1$ or $0$, and that according to~\eqref{c6} we have
\begin{align}\label{sum:ell}
	\ell_1 + \cdots + \ell_k = \#\{ j: \ i-1 \le j \le I-1 \text{ and } \tau(s_j) = \partial \}.
\end{align}
Our next step will be to detail the case $k=1$ in~\eqref{c6}. That is we claim that there exists $C<\infty$ depending only on $H$ such that
\begin{align}\label{claim}
	U(I_0, P_0, J_0^*, \theta^{(0)}) \le C^q T^{2-\ell_1} \,  U(I-3,P_1, J^*_1, \theta^{(1)}).
\end{align}
Before that we will explicit our methodology on an example.

\smallskip

\noindent
{\it Step 5: Example of integration.}
In order to make our integration procedure more concrete, let us see what we obtain following Example \ref{ex:I:bd}. In this case, recall that we consider the expression \eqref{U:ex}, and Step 4 consists in bounding the integral with respect to $s_4$. There is no power term in $s_4$ in \eqref{U:ex}, which means that we are in the context of inequality \eqref{c5}. We get
\begin{align}\begin{split}\label{U:step5}
&U\left( I=4, P = \big[\{1,3\}\big], J^* = \{2,4\}, \theta = (0,0,0,0) \right)\\
& \lesssim \int_{0<s_1<s_2<s_3<T} \exp\left( -K^4 (s_2-s_1)^{2H} \right) (s_3-s_1)^{2H-2} ds_1 ds_2 ds_3,
\end{split}\end{align}
which corresponds in \eqref{para:k=0} to
\begin{align*}
	I_0 = 3, \quad P_0 = \big[ \{1,3\} \big], \quad J_0^* = \{2\}, \quad \theta^{(0)} = (0,0,0).
\end{align*}
Now, anticipating on Step 6, observe that we just have $\nu=1$ iteration left. We consider $i=I_{k-1}$ for $k=1$, that is $i=3$. Moreover, looking at the right hand side of \eqref{U:step5}, the most singular term is $(s_3-s_1)^{2H-2}$. However, this term is neutralized by the fact that we first integrate in $s_3 \in (s_2,T)$. This yields
\begin{align*}
	&\int_{0<s_1<s_2<s_3<T} \exp\left( -K^4 (s_2-s_1)^{2H} \right) (s_3-s_1)^{2H-2} ds_1 ds_2 ds_3\\
	&\lesssim \int_{0<s_1<s_2<T} \exp\left( -K^4(s_2-s_1)^{2H} \right) (s_2-s_1)^{2H-1} ds_1 ds_2\\
	&\le \int_{0<s_1<s_2<T} (s_2-s_1)^{2H-1} ds_1 ds_2
	\lesssim T^{1+2H} \le T^2.
\end{align*}
This removal of singularities owing to the fact that singular variables are not contiguous in the simplex will be ubiquitous in the sequel. The current case corresponds to the estimate~\eqref{g} below (see also \eqref{c2.1}).

\noindent
{\it Step 6: Analyzing \eqref{claim}.}
In order to prove \eqref{claim}, we reduce the number of cases similarly to what we did for the $s_{I}$ variable. Specifically we
recall property~\ref{it:sigma-i} and the type $\tau(s_i)$ of variables defined above, and use property~\ref{it:sigma-ii} to deduce that
there are only three possibilities for the types of two consecutive variables $s_{i-1}$ and $s_i$:
\begin{enumerate}[label=\textbf{(\arabic*)}]
\item 
$\tau(s_{i-1})=\partial$, $\tau(s_i)=0$;
\item
$\tau(s_{i-1})=0$, $\tau(s_i)=\partial$;
\item
$\tau(s_{i-1})=0$, $\tau(s_i)=0$,
\end{enumerate}
where we note that the case $\tau(s_{i-1})=\partial$, $\tau(s_i)=\partial$ is ruled out by property~\ref{it:sigma-ii}. We now further analyze those 3 cases separately.

\smallskip

\noindent
{\bf Case (1)}: $\tau(s_{i-1})=\partial$, $\tau(s_i)=0$. In this case, $\ell_1 = 1$.
There are two subcases (here recall that $P_{0}$ is defined by~\eqref{para:k=0}):
\begin{enumerate}
\item[(1a)] $i-1 \in B$ for some $B \in P_0$ with $\# B = 1$;
\item[(1b)] $i-1 \in B$ for some $B \in P_0$ with $\# B = 2$.
\end{enumerate}
First consider case (1a).
Thanks to property~\ref{it:sigma-iv} above, at least one of $i, i-1$ belongs to $J^*$.
If $i \in J^*$, the factor $e^{-K^q(s_i-s_{i-1})^{2H}}$ appears in $U(I_0, P_0, J_0^*, \theta^{(0)})$ but the factor $e^{-K^q(s_{i-1}-s_{i-2})^{2H}}$ may or may not appear (depending on whether or not $i-1 \in J^*$). In any case, the latter is bounded by 1, hence the following integral appears in $U(I_0, P_0, J_0^*, \theta^{(0)})$ and is bounded by
\begin{align}\label{c61}
	&\int_{s_{i-2}}^T ds_{i-1} \int_{s_{i-1}}^T ds_i \, |s_{i-1}-s_{i-1+\theta_{i-1}}|^{2H-1} e^{-K^q(s_i-s_{i-1})^{2H}} \notag\\
	&\le\int_{s_{i-2}}^T ds_{i-1} \int_{s_{i-1}}^T ds_i \, \left[ (s_i-s_{i-1})^{2H-1} + (s_{i-1}-s_{i-2})^{2H-1} \right] e^{-K^q(s_i-s_{i-1})^{2H}} \notag\\
	& \le \int_0^T ds_{i-1}  \int_0^\infty ds_i\, s_i^{2H-1} e^{-K^qs_i^{2H}} + \int_0^T ds_{i-1} \, s_{i-1}^{2H-1} \int_0^\infty ds_i \, e^{-K^qs_i^{2H}} \notag\\
	& \le C^q (T + T^{2H})
	\lesssim C^q T = C^q T^{2-\ell_1} 
\end{align}
for some $C<\infty$ depending only on $H$, under the conditions that $H \in (0,1/2)$ and $T\ge 1$.
Similarly, if $i-1 \in J^*$, then the following integral appears and is bounded by
\begin{align}
	&\int_{s_{i-2}}^T ds_{i-1} \int_{s_{i-1}}^T ds_i \, |s_{i-1}-s_{i-1+\theta_{i-1}}|^{2H-1} e^{-K^q(s_{i-1}-s_{i-2})^{2H}} \notag\\
	&\le \int_{s_{i-2}}^T ds_{i-1} \int_{s_{i-1}}^T ds_i \, \left[ (s_i-s_{i-1})^{2H-1} + (s_{i-1}-s_{i-2})^{2H-1} \right] e^{-K^q(s_{i-1}-s_{i-2})^{2H}} \notag\\
	& \le \int_0^\infty ds_{i-1} \,e^{-K^qs_{i-1}^{2H}} \int_0^T ds_i\, s_i^{2H-1} + \int_0^\infty ds_{i-1} \, s_{i-1}^{2H-1} e^{-K^qs_{i-1}^{2H}} \int_0^T ds_i \notag\\
	& \le C^q (T^{2H} + T)
	\lesssim C^q T = C^q T^{2-\ell_1}. \label{c7}
\end{align}
For case (1b), we need to consider $B$ of the form $B = \{j, i-1\}$ with $j<i-1$, and an integral of the form
\begin{align}\label{J:1b}
	\mathcal{J}_{\text{1b}}
	= \int_{s_{i-2}}^T ds_{i-1} \int_{s_{i-1}}^T ds_i \, (s_{i-1}-s_j)^{2H-2}.
\end{align}
Along the same lines as for \eqref{c2.1}, the singularity on the diagonal is avoided in \eqref{J:1b} thanks to the fact that we actually have $j<i-2$. Hence, if we trivially bound the $s_i$-integral by $T$, we get
\begin{align*}
	\mathcal{J}_{\text{1b}} \lesssim T (s_{i-2}-s_j)^{2H-1} \le T (s_{j+1} - s_{j})^{2H-1}
	= T^{2-\ell_1} |s_j-s_{j+\theta_j^{(1)}}|^{2H-1}.
\end{align*}
The appearance of the last factor $|s_j-s_{j+\theta_j^{(1)}}|^{2H-1}$ justifies the update of $\theta$ in \eqref{theta:update}.

\smallskip

\noindent
{\bf Case (2)}: $\tau(s_{i-1})=0$, $\tau(s_i)=\partial$. 
Again, $\ell_1 = 1$. Similarly to the first case, we split the study into two subcases:
\begin{enumerate}
\item[(2a)] $i \in B$ for some $B \in P_0$ with $\# B = 1$;
\item[(2b)] $i \in B$ for some $B \in P_0$ with $\# B = 2$.
\end{enumerate}
In addition, we first assume that the following condition is met:
\begin{align}\label{i-2:+}
\{i-2\} \in P_0,\quad \theta^{(0)}_{i-2}=1.
\end{align}
Under this condition, the factor $(s_{i-1} - s_{i-2})^{2H-1}$ appears in $U(I_0,P_0,J_0^*,\theta^{(0)})$.
Recall that $i = I-1$ in the first iteration ($k=1$) and $\theta^{(0)}_i = \theta^{(0)}_{I-1}$ is either $-1$ or $0$; see \eqref{para:k=0} above.
Then for case (2a), we have $\theta_i^{(0)} = -1$.
If $i \in J^*$, then the following integral appears and is bounded (along the same lines as~\eqref{c7}) by
\begin{align*}
	&\int_{s_{i-2}}^T ds_{i-1} \int_{s_{i-1}}^T ds_i\, (s_{i-1}-s_{i-2})^{2H-1}(s_i-s_{i-1})^{2H-1} e^{-K^q(s_{i}-s_{i-1})^{2H}}\\
	&\le C^q\int_{s_{i-2}}^T ds_{i-1}\,  (s_{i-1}-s_{i-2})^{2H-1}
	\lesssim C^q T^{2H} \le C^q T = C^q T^{2-\ell_1};
\end{align*}
Similarly, if $i-1 \in J^*$, then we have to bound an integral of the form
\begin{align*}
	&\int_{s_{i-2}}^T ds_{i-1} \int_{s_{i-1}}^T ds_i\, (s_{i-1}-s_{i-2})^{2H-1}(s_i-s_{i-1})^{2H-1} e^{-K^q(s_{i-1}-s_{i-2})^{2H}}\\
	&\lesssim T^{2H} \int_{s_{i-2}}^T ds_{i-1} \, (s_{i-1}-s_{i-2})^{2H-1} e^{-K^q(s_{i-1}-s_{i-2})^{2H}}
	 \le C^q T^{2H} \le C^q T = C^q T^{2-\ell_1}.
\end{align*}
The removal of the factor $(s_{i-1} - s_{i-2})^{2H-1}$ justifies the update of $\theta$ in \eqref{theta:update}.


Let us now turn to case (2b) above. We assume  that $B=\{j, i\}$ with $j<i$. This hypothesis and condition \eqref{i-2:+} together imply that $j<i-2$.
In this case, we may bound any possible exponential factors by 1 and estimate the integral as follows:
\begin{align}\label{c8}
	&\int_{s_{i-2}}^T ds_{i-1} \int_{s_{i-1}}^T ds_i\, (s_{i-1}-s_{i-2})^{2H-1}(s_i-s_{j})^{2H-2}
\notag\\
	&\lesssim \int_{s_{i-2}}^T ds_{i-1} \, (s_{i-1}-s_{i-2})^{2H-1}(s_{i-1}-s_{j})^{2H-1}
	\le \int_{s_{i-2}}^T ds_{i-1} \, (s_{i-1}-s_{i-2})^{2H-1}(s_{j+1}-s_{j})^{2H-1}
\notag\\
	&\lesssim T^{2H} (s_{j+1}-s_j)^{2H-1}
	 \le T (s_{j+1}-s_j)^{2H-1} = T^{2-\ell_1} |s_j - s_{j +\theta_j^{(1)}}|^{2H-1}.
\end{align}
Again, the appearance of the last factor $|s_j - s_{j +\theta_j^{(1)}}|^{2H-1}$ is consistent with the update of $\theta$ in \eqref{theta:update}.

So far we have examined cases (2a)-(2b) assuming~\eqref{i-2:+} holds true. Next, suppose~\eqref{i-2:+} is false.
Then under case (2a), we have again $\theta^{(0)}_i = -1$.
If $i \in J^*$, then similarly to~\eqref{c61} we have to handle an integral of the form
\begin{align*}
	\int_{s_{i-2}}^T ds_{i-1} \int_{s_{i-1}}^T ds_i\, (s_i-s_{i-1})^{2H-1} e^{-K^q(s_i-s_{i-1})^{2H}} \le \int_{s_{i-2}}^T C^q ds_{i-1} \le C^q T = C^q T^{2-\ell_1};
\end{align*}
In addition, if $i \not\in J^*$ then we must have $i-1 \in J^*$. Hence we can proceed as in~\eqref{c7} with an integral which is written as:
\begin{align*}
	\int_{s_{i-2}}^T ds_{i-1} \int_{s_{i-1}}^T ds_i\, (s_i-s_{i-1})^{2H-1} e^{-K^q(s_{i-1}-s_{i-2})^{2H}}
	\le C^q T^{2H}\le C^q T = C^q T^{2-\ell_1}.
\end{align*}
We now deal with the case (2b) when ~\eqref{i-2:+} is false. In this situation we have $B=\{j,i \}$ with $j<i$.
If $j < i-2$, then by bounding any possible exponential factors by 1, the following integral appears and is bounded exactly as in~\eqref{c8}:
\begin{align}\begin{split}\label{c8}
	&\int_{s_{i-2}}^T ds_{i-1} \int_{s_{i-1}}^T ds_i \, (s_i-s_j)^{2H-2}	
	\lesssim \int_{s_{i-2}}^T ds_{i-1}\, (s_{i-1}-s_j)^{2H-1}\\
	& \le \int_{s_{i-2}}^T ds_{i-1}\, (s_{j+1}-s_j)^{2H-1} \le T(s_{j+1}-s_j)^{2H-1} = T^{2-\ell_1} |s_j - s_{j+\theta^{(1)}_j}|^{2H-1};
\end{split}\end{align}
When condition~\eqref{i-2:+} is not fulfilled, we cannot rule out the case $j = i-2$. However, if we have $j = i-2$, we can still avoid bad singularities on the diagonal. Specifically, the following integral appears and is bounded by

\begin{equation}\label{g}
\int_{s_{i-2}}^T ds_{i-1} \int_{s_{i-1}}^T ds_i\, (s_i-s_{i-2})^{2H-2}
	\lesssim \int_{s_{i-2}}^T ds_{i-1}\, (s_{i-1}-s_{i-2})^{2H-1} \lesssim T^{2H} \le T = T^{2-\ell_1}.
\end{equation}

\smallskip

\noindent
{\bf Case (3)}: $\tau(s_{i-1})=0$, $\tau(s_i)=0$. In this case, the quantity $\ell_{1}$ in~\eqref{c6} is such that $\ell_1 = 0$.
If \eqref{i-2:+} holds, then by bounding any possible exponential factor by 1, we have
\begin{align}\label{f1}
	\int_{s_{i-2}}^T ds_{i-1} \int_{s_{i-1}}^T ds_i \, (s_{i-1}-s_{i-2})^{2H-1} \lesssim T^{2H+1} \lesssim T^2 = T^{2-\ell_1}.
\end{align}
If \eqref{i-2:+} is false, notice that this corresponds to a trivial case with no differentiation.
The corresponding bound is of the form
\begin{align}\label{f2}
	\int_{s_{i-2}}^T ds_{i-1} \int_{s_{i-1}}^T ds_i \le T^2 = T^{2-\ell_1}.
\end{align}
Summarizing this step, our inequalities \eqref{c61}--\eqref{f2} give a bound on all the possible factors allowing to go from $U(I_0,P_0,J_0^*,\theta^{(0)})$ to $U(I-3,P_1,J_1^*,\theta^{(1)})$, with $\theta^{(1)}$ obeying the rule \eqref{theta:update}.
In other words, we have proved \eqref{claim}.

\smallskip

\noindent
\textit{Step 7: General iteration step.}
We now wish to extend our Step 6, which established a relation between $U(I-1,P_0,J_0^*,\theta^{(0)})$ and $U(I-3,P_1,J_1^*,\theta^{(1)})$, to the generic $k$-th integration over the simplex.
Recalling the notation \eqref{c6}, one can repeat the same computations as in {\bf Cases (1)--(3)} above, except for the fact that we are now integrating the variables $s_{I_{k-1}-1}, s_{I_{k-1}}$. We let the reader check that for $k=\{2,\dots,\nu-1\}$ we obtain
\begin{align}\label{iteration}
	U(I_{k-1}, P_{k-1}, J^*_{k-1}, \theta^{(k-1)}) \le C^q T^{2-\ell_k} \,  U(I_k,P_k, J^*_k, \theta^{(k)}).
\end{align}

\smallskip

\noindent
\textit{Step 8: Conclusion.}
Putting together \eqref{U:I_0} and \eqref{claim}, we have obtained
\begin{align}\begin{split}\label{h}
	U(I,P,J^*,\theta) & \le C^q U(I_0, P_0, J_0^*, \theta^{(0)})\\
	&\le C^{2q} T^{2-\ell_1} U(I_1,P_1,J_1^*,\theta^{(1)}).
\end{split}\end{align}
We can now proceed to apply the general case \eqref{iteration}, separating two cases.

\begin{enumerate}[label=\textbf{(\roman*)}]
\item
If $I$ is odd, that is, if $I_0$ is even. In that case we simply use \eqref{iteration} iteratively starting from the right hand side of \eqref{h}.
We end up with
\begin{align}\begin{split}\label{i}
	U(I,P,J^*,\theta)
	&\le C^{2q} T^{2-\ell_1} U(I_1,P_1,J^*_1,\theta^{(1)})\\
	& \,\,\,\vdots\\
	& \le C^{q\nu} T^{2(\nu-1) - (\ell_1+\cdots+\ell_{\nu-1})} U(2, P_{\nu-1}, J^*_{\nu-1},\theta^{(\nu-1)})\\
	& \le C^{q(\nu+1)} T^{2\nu - (\ell_1+\cdots+\ell_\nu)}.\end{split}\end{align}
Recall that $T \ge 1$.
Also, note that $2\nu = I-1$, and observe from \eqref{sum:ell} and property \ref{it:sigma-ii} above that
\begin{align*}
	\ell_1+\cdots+\ell_\nu = \#\{ j \le I-1: \, \tau(s_j) = \partial \}
	= \#\{ j \le I: \, \tau(s_j) = \partial \} =\ell.
\end{align*}
Plugging this information into \eqref{i}, it follows that
\begin{align*}
	U(I,P,J^*,\theta)
	\le C^{qI} T^{I-1 - \ell} \le C^{q^2} T^{I-\ell}.
\end{align*}
\item
If $I$ is even, that is, if $I_0$ is odd. In this situation we can still start from \eqref{h} and invoke~\eqref{iteration} successively.
Similarly to \eqref{i}, we get
\begin{align}\begin{split}\label{k}
	U(I,P,J^*,\theta)
	&\le C^{2q} T^{2-\ell_1} U(I_1,P_1,J^*_1,\theta^{(1)})\\
	& \,\,\,\vdots\\
	& \le C^{q\nu} T^{2(\nu-1) - (\ell_1+\cdots+\ell_{\nu-1})} U(3, P_{\nu-1}, J^*_{\nu-1},\theta^{(\nu-1)})\\
	& \le C^{q(\nu+1)} T^{2\nu - (\ell_1+\cdots+\ell_\nu)} U(1,P_\nu, J_\nu^*, \theta^{(\nu)}).
\end{split}\end{align}
The slight difference here is that we still have to evaluate $U(1,P_\nu,J_\nu^*,\theta^{(\nu)})$ above.
Specifically, owing to the definition \eqref{nu} of $\nu$, we still have one variable to integrate in $U(1,P_\nu,J_\nu^*,\theta^{(\nu)})$.
On the other hand, all singularities have been removed thanks to the repeated applications of \eqref{iteration}.
We thus discover that
\begin{align}\label{l}
	U(1,P_\nu, J_\nu^*, \theta^{(\nu)}) \le \int_0^T 1 \, ds \le T.
\end{align}
Furthermore, observe that
\begin{align}\label{m}
	\ell_1+\cdots+\ell_\nu = \#\{ j: \ 2 \le j \le I, \tau(s_j) = \partial \} = 	\begin{cases}
		\ell & \text{if $\tau(s_1) = 0$;}\\
		\ell-1 & \text{if $\tau(s_1) = \partial$,}
	\end{cases}
\end{align}
and $2\nu = I-2$.
Hence gathering \eqref{l} and \eqref{m} into \eqref{k}, it follows that
\begin{align}\label{n}
	U(I,P,J^*,\theta)
	\le C^{q(\nu+2)} T^{I-1 - (\ell_1+\cdots+\ell_\nu)} \le C^{q^2} T^{I-\ell}.
\end{align}
\end{enumerate}
Now, we may combine the odd case (i) and the even case (ii) to obtain \eqref{U:UB}.
Then, putting~\eqref{U:UB} back into~\eqref{I:UB} yields $|\mathcal{I}[\sigma; \varepsilon, G_\varepsilon]| \le C^{q^2} T^{2q-m-\ell}$.
Finally, recall the relation $m+\ell=q$ (see \eqref{m+l=q} above) to finish the proof.
\end{proof}

We now gather everything to prove the main result.

\begin{proof}[Proof of Theorem \ref{th:main}]
This result has been proved for $H=1/2$ by \cite{FS18} and for $H \in (1/2,1)$ by \cite{LL25}.
It remains to prove it for $H\in (0,1/2)$.
We will prove it by verifying the hypothesis of Proposition \ref{pr:Fdim}.
Define $\nu$ as the image measure $\nu(A) = m\left\{ t \in [0,1]: B_t \in A \right\}$ as alluded to in Section \ref{s:reduction}.
It is known \cite{K85} that
\begin{align}\label{E:vertical}
	\E\left[ |\hat{\nu}(\xi_2)|^{2q} \right] \lesssim |\xi_2|^{-q/H} \quad \text{for all $\xi_2 \ne 0$ and $q \in \N_+$.}
\end{align}
We thus focus on the horizontal bounds of Proposition \ref{pr:Fdim}.
To this aim, for $H \in (0,1/2)$, we may apply Lemma \ref{lem:I:UB} to \eqref{I[eps,G]} and put it into \eqref{2q:moment} to deduce that there exists $C_q<\infty$ such that
\begin{align*}
	\E\left[ |\hat{\mu}_G(\xi_1,\xi_2)|^{2q} \right]
	= \frac{(q!)^2}{|\xi_2|^{2q/H}} \sum_{\varepsilon\in \mathcal{A}_{2q}} \left( \frac{3}{2\pi \lambda} \right)^q C^{q^2} T^q
	\le \frac{C_qT^q}{|\xi_2|^{2q/H} |\lambda|^q}
\end{align*}
for all $q\in \N_+$, $\xi_1\ne 0$ and $\xi_2 \ne 0$. Recall \eqref{lambda,T} for the values of $\lambda$ and $T$, and use this to obtain from the above that
\begin{align}\label{E:horizontal}
	\E\left[ |\hat{\mu}_G(\xi_1,\xi_2)|^{2q} \right] \le C_q|\xi_1|^{-q} \quad \text{for all $\xi_1\ne 0$ and $q \in \N_+$.}
\end{align}
Hence, the estimates \eqref{E:vertical} and \eqref{E:horizontal} together allow us to apply Proposition \ref{pr:Fdim} and deduce that $\dim_F G(B) \ge \min\{H^{-1}, 1\}$ a.s. Since $H <1/2$, it follows that $\dim_F G(B) \ge 1$. But $B_t$ is continuous, so we know that $\dim_F G(B) \le 1$ (see \cite{FOS14}).
This completes the proof.
\end{proof}

\section*{Acknowledgements}  Cheuk Yin Lee is supported by a research startup fund of the Chinese University of Hong Kong, Shenzhen and the Shenzhen Peacock fund 2025TC0013. Samy Tindel is supported for this work by the NSF grant DMS-2450734.

\Addresses

\end{document}